\documentclass{birkjour}
 \newtheorem{thm}{Theorem}[section]
 \newtheorem{cor}[thm]{Corollary}
 \newtheorem{lem}[thm]{Lemma}
 
 \theoremstyle{definition}
 \newtheorem{defn}[thm]{Definition}
 \theoremstyle{rem}
 \newtheorem{rem}[thm]{Remark}
 
 \numberwithin{equation}{section}
 
\newcommand{\bc}{\begin{center}}
\newcommand{\ec}{\end{center}}
\newcommand{\bt}{\begin{tabular}}
\newcommand{\et}{\end{tabular}} 
\newcommand{\bea}{\begin{eqnarray}}
\newcommand{\eea}{\end{eqnarray}}
\newcommand{\bean}{\begin{eqnarray*}}
\newcommand{\eean}{\end{eqnarray*}}

\newcommand{\ba}{\begin{array}}
\newcommand{\ea}{\end{array}}

\def\be{\begin{eqnarray}}
\def\ee{\end{eqnarray}}
\def\ben{\begin{eqnarray*}}
\def\een{\end{eqnarray*}}

\newcommand{\ra} {\rightarrow}

\newcommand{\nth}{\frac{1}{n}}

\newcommand{\RL}{{\mathbb R}}

\newcommand{\VAR}{\mbox{\rm Var}}

\newcommand{\lam}{\lambda}

\def\elabel#1{\label{e:#1}}

\def\sq{$\Box$}

\def\qed{\ifmmode\sq\else{\unskip\nobreak\hfil
\penalty50\hskip1em\null\nobreak\hfil\sq
\parfillskip=0pt\finalhyphendemerits=0\endgraf}\fi\par\medbreak}

\newsavebox{\junk}
\savebox{\junk}[1.6mm]{\hbox{$|\!|\!|$}}

\def\til={{\widetilde =}}

 \def\eq#1/{(\ref{#1})}

\def\eq#1/{(\ref{e:#1})}

\newcommand{\beqn}[1]{\notes{#1}%
\begin{eqnarray} \elabel{#1}}

\newcommand{\eeqn}{\end{eqnarray} }

\newcommand{\beq}[1]{\notes{#1}%
\begin{equation}\elabel{#1}}

\newcommand{\eeq}{\end{equation}} 

\def\bdes{\begin{description}}
\def\edes{\end{description}}

\def\notes#1{}

\def\L{L_\infty}

\def\h{{\widetilde h}}
\def\Var{{\rm Var}}
\newcommand{\R}{\RL}
\newcommand{\E}{\mathbf{E}}
\renewcommand{\P}{\mathbf{P}}

\def\bee{\begin{eqnarray*}}
\def\ene{\end{eqnarray*}}

\begin{document}

%
%
%
%
%
%
%
%
%

\title[Concentration of Information]
 {Optimal Concentration of Information Content for Log-Concave Densities}

\author[Fradelizi]{Matthieu~Fradelizi}

\address{%
Universit\'e Paris-Est Marne-la-Vall\'ee\\
Laboratoire d'Analyse et de Mathématiques Appliquées UMR 8050\\
5 Bd Descartes, Champs-sur-Marne\\
77454 Marne-la-Vallée Cedex 2\\
France}

\email{matthieu.fradelizi@univ-mlv.fr}

\thanks{This work was partially supported by the project GeMeCoD ANR 2011 BS01 007 01,
and by the U.S. National Science Foundation through the grant DMS-1409504 (CAREER).
A significant portion of this paper is based on the Ph.D. dissertation of L. Wang \cite{Wan14:phd}, co-advised by M. Madiman
and N. Read, at Yale University}
\author[Madiman]{Mokshay~Madiman}
\address{University of Delaware\\
Department of Mathematical Sciences\\
501 Ewing Hall \\
Newark, DE 19716\\
USA}
\email{madiman@udel.edu}
\author[Wang]{Liyao~Wang}
\address{J. P. Morgan\\
New York NY\\
USA}
\email{njuwangliyao@gmail.com}


\date{March 1, 2015}

\begin{abstract}
An elementary proof is provided of sharp bounds for the varentropy of
random vectors with log-concave densities, as well as for deviations of the
information content from its mean. These bounds significantly improve on the bounds
obtained by Bobkov and Madiman ({\it Ann. Probab.}, 39(4):1528--1543, 2011).
\end{abstract}

\maketitle
\section{Introduction}
\label{sec:intro}

Consider a random vector $Z$ taking values in $\RL^n$, drawn from the standard Gaussian distribution $\gamma$,
whose density is given by
\ben
\phi(x) = \frac{1}{(2\pi)^{\frac{n}{2}}} e^{-\frac{|x|^2}{2}}
\een
for each $x\in\RL^n$, where $|\cdot|$ denotes the Euclidean norm.
It is well known that when the dimension $n$ is large, the distribution of $Z$ is highly concentrated around the sphere of radius $\sqrt{n}$;
that $\sqrt{n}$ is the appropriate radius follows by the trivial observation that $\E |Z|^2= \sum_{i=1}^n \E Z_i^2 = n$.
One way to express this concentration property is by computing the variance of $|Z|^2$, which is easy to do using the independence
of the coordinates of $Z$:
\ben\begin{split}
\Var(|Z|^2) 
&= \Var\bigg( \sum_{i=1}^n Z_i^2 \bigg) 
= \sum_{i=1}^n \Var(Z_i^2)
= 2n .
\end{split}\een
In particular, the standard deviation of $|Z|^2$ is $\sqrt{2n}$, which is much smaller than the mean $n$ of $|Z|^2$
when $n$ is large. Another way to express this concentration property is through a deviation inequality:
\be\label{eq:chernoff}
\P \bigg\{ \frac{|Z|^2}{n} -1 >t  \bigg\} \leq  \exp\bigg\{ -\frac{n}{2} [t-\log(1+t)] \bigg\}
\ee
for the upper tail, and a corresponding upper bound on the lower tail. These inequalities immediately follow
from Chernoff's bound, since $|Z|^2/n$ is just the empirical mean of i.i.d. random variables.

It is natural to wonder if, like so many other facts about Gaussian measures, the above concentration property also
has an extension to log-concave measures (or to some subclass of them). There are two ways one may think about
extending the above concentration property. One is to ask if there is a universal constant $C$ such that
\ben
\Var(|X|^2) \leq Cn ,
\een
for every random vector $X$ that has an isotropic, log-concave distribution on $\RL^n$. Here, we say that a distribution
on $\RL^n$ is isotropic if its covariance matrix is the identity matrix; this assumption ensures that $\E |X|^2 =n$, and
provides the normalization needed to make the question meaningful. This question has been well studied in the literature,
and is known as the ``thin shell conjecture'' in convex geometry. It is closely related to other famous conjectures: it implies
the hyperplane conjecture of Bourgain \cite{EK11, EL14}, is trivially implied by the Kannan-Lovasz-Simonovits conjecture,
and also implies the Kannan-Lovasz-Simonovits conjecture up to logarithmic terms \cite{Eld13}.
The best bounds known to date are those of Gu\'edon and E.~Milman \cite{GM11}, and assert that
\ben
\Var(|X|^2) \leq Cn^{4/3} .
\een

The second way that one may try to extend the above concentration property from Gaussians to log-concave measures
is to first observe that the quantity that concentrates, namely $|Z|^2$, is essentially the logarithm of the Gaussian density function.
More precisely, since
\ben
-\log\phi(x) = \frac{n}{2}\log(2\pi) + \frac{|x|^2}{2} ,
\een
the concentration of $|Z|^2$ about its mean is equivalent to the concentration of $-\log\phi(Z)$ about its mean.
Thus one can ask if, for every random vector $X$ that has a log-concave density $f$ on $\RL^n$,
\be\label{eq:varbd}
\Var(-\log f(X)) \leq Cn
\ee
for some absolute constant $C$. An affirmative answer to this question was provided by Bobkov and Madiman \cite{BM11:aop}.
The approach of \cite{BM11:aop} can be used to obtain bounds on $C$, 
but the bounds so obtained are quite suboptimal (around 1000). Recently V.~H.~Nguyen \cite{Ngu13:phd} (see also \cite{Ngu14:1}) and L.~Wang \cite{Wan14:phd}
independently determined, in their respective Ph.D. theses, that the sharp constant $C$ in the bound \eqref{eq:varbd}  is 1. 
Soon after this work, simpler proofs of the sharp variance bound were obtained independently by
us (presented in the proof of Theorem~\ref{thm:varent} in this paper) and by Bolley, Gentil and Guillin \cite{BGG15}
(see Remark 4.2 in their paper). An advantage of our proof over the others mentioned 
is that it is very short and straightforward, and emerges as a consequence of a more basic 
log-concavity property (namely Theorem~\ref{thm:lc}) of $L^p$-norms of log-concave functions,
which may be thought of as an analogue for log-concave functions of a classical inequality of Borell \cite{Bor73a} for
concave functions.

If we are interested in finer control of the integrability of $-\log f(X)$, we may wish to consider
analogues for general log-concave distributions of the inequality \eqref{eq:chernoff}. Our
second objective in this note is to provide such an analogue (in Theorem~\ref{thm:lc-dev}). A weak version of such a statement
was announced in \cite{BM11:cras} and proved in \cite{BM11:aop}, but the bounds we provide in this note are much stronger.
Our approach has two key advantages: first, the proof is transparent and completely avoids the use of the
sophisticated Lovasz-Simonovits localization lemma, which is a key ingredient of the approach
in \cite{BM11:aop}; and second, our bounds on the moment generating function are sharp, and are attained 
for example when the distribution under consideration has i.i.d. exponentially distributed marginals.

While in general exponential deviation inequalities imply variance bounds, the reverse
is not true. Nonetheless, our approach in this note is to first prove the variance bound
\eqref{eq:varbd}, and then use a general bootstrapping result (Theorem~\ref{thm:mainconcentrate})
to deduce the exponential deviation inequalities from it. The bootstrapping result is of independent interest;
it relies on a technical condition that turns out  to be automatically satisfied when
the distribution in question is log-concave. 

Finally we note that many of the results in this note can be extended to the class of
convex measures; partial work in this direction is done by Nguyen \cite{Ngu14:1},
and results with sharp constants are obtained in the forthcoming paper \cite{FLM15}.


\section{Optimal varentropy bound for log-concave distributions}
\label{sec:var}

%
%

Before we proceed, we need to fix some definitions and notation.

\begin{defn}
Let a random vector $X$ taking values in $\RL^n$ have probability density
function $f$.
The {\it information content} of $X$ is the random variable $\h(X)=-\log f(X)$.
The {\it entropy} of $X$ is defined as \, $h(X)\,=\,\E(\h(X))$.
The {\it varentropy} of a random vector $X$ is defined as \, $V(X)\,=\,\VAR(\h(X))$.
\end{defn}

Note that the entropy and varentropy depend not on the realization of $X$
but only on its density $f$,
whereas the information content does indeed depend on the realization of $X$.
For instance, one can write $h(X)=-\int_{\RL^n} f\log f$ and 
\ben
V(X)= \Var(\log f(X))= \int_{\RL^n} f(\log f)^2 - \bigg(\int_{\RL^n} f\log f\bigg)^2 .
\een
Nonetheless, for reasons of convenience and in keeping with historical convention,
we slightly abuse notation as above.

As observed in \cite{BM11:aop}, the distribution of the difference $\h(X) - h(X)$
is invariant under any affine transformation of $\RL^n$
(i.e., $\h(TX) - h(TX) = \h(X) - h(X)$
for all invertible affine maps $T:\R^n \rightarrow \R^n$);
hence the varentropy $V(X)$ is affine-invariant while the entropy $h(X)$ is not.

Another invariance for both $h(X)$ and $V(X)$ follows from the fact that they only depend on the 
distribution of $\log(f(X))$, so that they are unchanged if $f$ is modified in such a way that its 
sublevel sets keep the same volume. This implies  (see, e.g., \cite[Theorem 1.13]{LL01:book}) that if $\tilde{f}$ is the 
spherically symmetric, decreasing rearrangement of $f$, 
and $\tilde{X}$ is distributed according to the density $\tilde{f}$, then 
$h(X)=h(\tilde{X})$ and $V(X)=V(\tilde{X})$. The rearrangement-invariance of
entropy was a key element in the development of refined entropy power
inequalities in \cite{WM14}.

Log-concavity is a natural shape constraint for functions (in particular, probability density functions)
that represents an infinite-dimensional generalization of the class of Gaussian distributions.

\begin{defn}
A function $f:\RL^n\ra [0,\infty)$ is {\it log-concave} if $f$ can be written as
\bee
f(x)=e^{-U(x)},
\ene
where $U : \RL^n \mapsto (-\infty, +\infty]$ is a convex function, i.e.,
$U(tx+(1-t)y)\leq tU(x)+(1-t)U(y)$
for any $x$, $y$ and $0<t<1$. When $f$ is a probability density function
and is log-concave, we say that $f$ is a {\it log-concave density}.
\end{defn}

We can now state the optimal form of the inequality \eqref{eq:varbd},
first obtained by Nguyen \cite{Ngu13:phd} and Wang \cite{Wan14:phd}
as discussed in Section~\ref{sec:intro}.

\begin{thm}\cite{Ngu13:phd, Wan14:phd}\label{thm:varent}
Given a random vector $X$ in $\R^n$ with log-concave density $f$,
$$
V(X)\leq n
$$
\end{thm}

\begin{rem}
The probability bound {\it does not depend} on $f$-- it is universal over the class
of log-concave densities.
\end{rem}

\begin{rem}\label{rem:extremal}
The bound is {\it sharp}. 
Indeed, let $X$ have density $f=e^{-\varphi}$, with $\varphi:\R^n\to [0,\infty]$ being positively homogeneous of degree 1, i.e., 
such that $\varphi(tx)=t\varphi(x)$ for all $t>0$ and all $x\in\R^n$. 
Then one can check that the random variable $Y=\varphi(X)$ has a gamma distribution
with shape parameter $n$ and scale parameter 1, i.e., it is distributed according to the density given by
\ben
f_Y(t)=\frac{t^{n-1}e^{-t}}{(n-1)!} .
\een
Consequently  $\E(Y)=n$ and $\E(Y^2)=n(n+1)$, and therefore $V(X)=\Var(Y)=n$.
Particular examples of equality include:
\begin{enumerate}
\item The case where $\varphi(x)=\sum_{i=1}^n x_i$ on the
cone of points with non-negative coordinates (which corresponds to $X$ having i.i.d. coordinates
with the standard exponential distribution), and 
\item The case where 
$\varphi(x)=
\inf\{r> 0: x\in rK \}$ for some compact convex set $K$ containing the origin
(which, by taking $K$ to be a symmetric convex body, includes all norms on $\R^n$ suitably normalized so that $e^{-\varphi}$ is a density).
\end{enumerate}
\end{rem}

\begin{rem}
Bolley, Gentil and Guillin \cite{BGG15} in fact prove a stronger inequality, namely,
\ben
\frac{1}{V(X)}- \frac{1}{n}\geq \bigg[\E \big\{ \nabla U(X) \cdot \text{Hess}(U(X))^{-1} \nabla U(X) \big\} \bigg]^{-1} .
\een
This gives a strict improvement of Theorem~\ref{thm:varent}
when the density $f=e^{-U}$ of $X$ is strictly log-concave, in the sense that 
$\text{Hess}(U(X))$  is, almost surely, strictly positive definite.
As noted by \cite{BGG15}, one may give another alternative proof of Theorem~\ref{thm:varent} 
by applying a result of Harg\'e \cite[Theorem 2]{Har08:2}.
\end{rem}

In order to present our proof of Theorem~\ref{thm:varent}, we will need
some lemmata. The first one is a straightforward computation
that is a special case of a well known fact about exponential families
in statistics, but we write out a proof for completeness.

\begin{lem}\label{lem:formula}
Let $f$ be any probability density function on $\RL^n$ such that $f\in L^\alpha(\RL^n)$
for each $\alpha>0$, and define
\ben
F(\alpha)=\log \int_{\RL^n} f^{\alpha} .
\een
Let $X_{\alpha}$ be a random variable with density $f_{\alpha}$
on $\RL^n$, where
\ben
f_{\alpha}:= \frac{f^{\alpha}}{\int_{\RL^n} f^{\alpha}} .
\een
Then $F$ is infinitely differentiable on $(0,\infty)$, and moreover,
for any $\alpha>0$,
\ben
F''(\alpha)=\frac{1}{\alpha^2} V(X_\alpha) .
\een
\end{lem}

\begin{proof}
Note that the assumption that $f\in L^\alpha(\RL^n)$ (or equivalently that
$F(\alpha)<\infty$) for all $\alpha>0$ guarantees that $F(\alpha)$ is infinitely differentiable for
$\alpha>0$ and that we can freely change the order of taking expectations and differentiation.

Now observe that
\ben\begin{split}
F'(\alpha)&= \frac{\int f^\alpha \log f}{\int f^\alpha}
= \int f_\alpha \log f  ;
\end{split}\een
if we wish, we may also massage this to write
\be\label{eq:1der}\begin{split}
F'(\alpha)
&= \frac{1}{\alpha} [ F(\alpha)- h(X_\alpha) ] .
\end{split}\ee
Differentiating again, we get
\ben\begin{split}
F''(\alpha)&= \frac{\int f^\alpha (\log f)^2}{\int f^\alpha}-\left(\frac{\int f^\alpha \log f}{\int f^\alpha}\right)^2\\
&= \int f_\alpha (\log f)^2 - \bigg( \int f_\alpha \log f\bigg)^2\\
&=\Var [  \log f(X_\alpha)]
= \Var \bigg[ \frac{1}{\alpha} \{ \log f_\alpha(X_\alpha) + F(\alpha)\} \bigg] \\
&= \frac{1}{\alpha^2} \Var [  \log f_\alpha(X_\alpha)]
= \frac{V(X_\alpha)}{\alpha^2},
\end{split}
\een
%
%
as desired.
\end{proof}

The following lemma is a standard fact about the so-called perspective function
in convex analysis. The use of this terminology is due to Hiriart-Urruty and Lemar\'echal \cite[p. 160]{HL93a:book}
(see \cite{BV04:book} for additional discussion), although the notion has been used without a name in 
convex analysis for a long time (see, e.g., \cite[p. 35]{Roc70:book}).
Perspective functions have also seen recent use in convex geometry \cite{CFPP14, BFM15, FLM15})
and empirical process theory \cite{VW11}. 
We give the short proof for completeness.

\begin{lem}\label{lem:cvx}
If $U:\RL^n\ra\RL\cup \{+\infty\}$ is a convex function, then
\ben
w(z,\alpha):=\alpha U(z/\alpha)
\een
is a  convex function on $\R^n\times(0,+\infty)$.
\end{lem}

\begin{proof}
First note that by definition, $w(az, a\alpha)=aw(z,\alpha)$ for any $a>0$ and any $(z,\alpha)\in \R^n\times(0,+\infty)$,
which implies in particular that
\ben
\frac{1}{\alpha} w(z,\alpha)= w\bigg(\frac{z}{\alpha}, 1\bigg) .
\een
Hence
\ben\begin{split}
& w(\lam z_1 +(1-\lam)z_2, \lam \alpha_1 +(1-\lam)\alpha_2) \\
&= [\lam \alpha_1 +(1-\lam)\alpha_2] \,
U\bigg( \frac{\lam \alpha_1 \frac{z_1}{\alpha_1} +(1-\lam)\alpha_2 \frac{z_2}{\alpha_2}}{\lam \alpha_1 +(1-\lam)\alpha_2} \bigg)\\
&\leq \lam \alpha_1 U\bigg(\frac{z_1}{\alpha_1}\bigg) + (1-\lam)\alpha_2 U\bigg(\frac{z_2}{\alpha_2}\bigg) \\
&= \lam w(z_1, \alpha_1) + (1-\lam) w(z_2, \alpha_2) ,
\end{split}\een
for any $\lam\in [0,1]$, $z_1, z_2 \in\RL^n$, and $\alpha_1, \alpha_2 \in (0,\infty)$.
\end{proof}


The key observation is the following theorem.

\begin{thm}\label{thm:lc}
If $f$ is log-concave on $\R^n$, then the function
\ben
G(\alpha):= \alpha^n\int f(x)^{\alpha} dx
\een
is log-concave on $(0,+\infty)$.
\end{thm}

\begin{proof}
Write $f=e^{-U}$, with $U$ convex. Make the change of variable $x=z/\alpha$ to get
\ben
G(\alpha)=\int e^{-\alpha U(z/\alpha)}dz .
\een
The function $w(z,\alpha):=\alpha U(z/\alpha)$ is convex on $\R^n\times(0,+\infty)$ by Lemma~\ref{lem:cvx},
which means that the integrand above is log-concave.
The log-concavity of $G$  then follows from Pr\'ekopa's theorem \cite{Pre73},
which implies that marginals of log-concave functions are log-concave.
\end{proof}

\begin{rem}
An old theorem of Borell \cite[Theorem 2]{Bor73a} states that if $f$ is concave on $\R^n$, then
$G_f(p):=(p+1)\cdots(p+n)\int f^p$is log-concave as a function of $p\in (0,\infty)$. Using this and the fact that a log-concave
function is a limit of $\alpha$-concave functions with $\alpha\to 0$, one can obtain
an alternate, indirect proof of Theorem~\ref{thm:lc}. One can also similarly obtain 
an indirect proof of Theorem~\ref{thm:lc} by considering a limiting version of \cite[Theorem VII.2]{BM11:it},
which expresses a log-concavity property of $(p-1)\ldots (p-n)\int \phi^{-p}$ for any convex function 
$\phi$ on $\R^n$, for $p>n+1$ (an improvement of this to the optimal range $p>n$ is described
in \cite{FLM15, BFM15}, although this is not required for this alternate proof of Theorem~\ref{thm:lc}).
\end{rem}

\begin{proof}[Proof of Theorem~\ref{thm:varent}]
Since $f$ is a log-concave density, it necessarily holds that $f\in L^\alpha(\RL^n)$ for every $\alpha>0$;
in particular, $G(\alpha):= \alpha^n\int f^{\alpha}$ is finite and infinitely differentiable on the domain $(0,\infty)$.
By definition,
\ben
\log G(\alpha)= n\log \alpha + \log \int f^{\alpha}
= n\log \alpha + F(\alpha) .
\een
Consequently,
\ben
\frac{d^2}{d \alpha^2} [\log G(\alpha)]
= -\frac{n}{\alpha^2} + F''(\alpha) .
\een
By Theorem~\ref{thm:lc},  $\log G(\alpha)$ is concave, and hence
we must have that
\ben
-\frac{n}{\alpha^2} + F''(\alpha) \leq 0
\een
for each $\alpha>0$. However, Lemma~\ref{lem:formula} implies that
$F''(\alpha)=V(X_\alpha)/\alpha^2$, so that we obtain the inequality
\ben
\frac{V(X_\alpha)-n}{\alpha^2} \leq 0 .
\een
For $\alpha=1$, this implies that $V(X)\le n$.
\end{proof}

Notice that if $f=e^{-U}$, where $U:\R^n\to[0,\infty]$ is positively homogeneous of degree 1, then 
the same change of variable as in the proof of Theorem~\ref{thm:lc} shows that 
$$G(\alpha)=\int e^{-\alpha U(z/\alpha)}dz=\int e^{-U(z)}dz=\int f(z)dz=1.$$
Hence the function $G$ is constant. Then the proof above shows that 
$V(X)=n$, which establishes the equality case stated in Remark~\ref{rem:extremal}.


\section{A general bootstrapping strategy}
\label{sec:gen}

The purpose of this section is to describe a strategy for obtaining exponential deviation inequalities
when one has uniform control on variances of a family of random variables. Log-concavity is not
an assumption made anywhere in this section.

\begin{thm}\label{thm:mainconcentrate}
Suppose $X\sim f$, where $f\in L^\alpha(\RL^n)$ for each $\alpha>0$.
Let $X_\alpha\sim f_\alpha$, where
\ben
f_\alpha(x) = \frac{f^\alpha(x)}{\int f^\alpha} .
\een
If $K=K(f):=\sup_{\alpha>0} V(X_\alpha)$, then
\ben
\E\big[e^{\beta\{\h(X)-h(X)\}}\big] \leq e^{Kr(-\beta)} , \quad \beta\in\RL ,
\een
where
\ben
r(u)= \left\{ \begin{array}{ll}
u-\log(1+u) & \text{ for } u>-1 \\
+\infty & \text{ for }  u\leq -1 \quad . \\
\end{array} \right.
\een
\end{thm}

\begin{proof}
Suppose $X$ is a random vector drawn from a density $f$ on $\RL^n$,
and define, for each $\alpha>0$,
$F(\alpha)=\log \int f^\alpha$.
Set
\ben
K=\sup_{\alpha>0} V(X_\alpha)= \sup_{\alpha>0} \alpha^2 F''(\alpha) ;
\een
the second equality follows from Lemma~\ref{lem:formula}.
Since $f\in L^\alpha(\RL^n)$ for each $\alpha>0$,  $F({\alpha})$ is finite
and moreover, infinitely differentiable for
$\alpha>0$, and we can freely change the order of integration and differentiation
when differentiating $F({\alpha})$.


%
From Taylor-Lagrange formula, for every $\alpha>0$, one has 
$$F(\alpha)=F(1) +(\alpha -1)F'(1)+\int_1^\alpha (\alpha-u)F''(u)du.$$
Using that $F(1)=0$, $F''(u)\le K/u^2$ for every $u>0$ and the fact that for $0<\alpha<u<1$, one has $\alpha-u<0$, we get
\ben\begin{split}
F(\alpha)
&\le (\alpha -1)F'(1)+K\int_1^\alpha \frac{\alpha-u}{u^2}du \\
&=(\alpha -1)F'(1)+K\left[-\frac{\alpha}{u}-\log(u)\right]_1^\alpha.
\end{split}\een
Thus, for $\alpha>0$, we have proved that
\bee
F(\alpha) \leq (\alpha-1) F'(1)+K(\alpha-1-\log\alpha).
\ene

Setting $\beta=1-\alpha$, we have for $\beta<1$ that
\be\label{eq:beta}
e^{F(1-\beta)} \leq e^{-\beta F'(1)}e^{K(-\beta-\log(1-\beta))}.
\ee
Observe that
$e^{F(1-\beta)}= \int f^{1-\beta} = \E [f^{-\beta}(X)] = \E [e^{-\beta \log f(X)}] =\E \big[e^{\beta\h(X)}\big]$
and
$e^{-\beta F'(1)}=e^{\beta h(X)}$; the latter fact follows from the fact that
$F'(1)=-h(X)$ as is clear from the identity \eqref{eq:1der}. Hence the inequality \eqref{eq:beta}
may be rewritten as
\be\label{eq:beta2}
\E\big[e^{\beta\{\h(X)-h(X)\}}\big] \leq e^{Kr(-\beta)} , \quad \beta\in\RL .
\ee
\end{proof}

\begin{rem}
We note that the function $r(t)=t-\log (1+t)$ for $t> -1$,
(or the related function $h(t)=t\log t-t+1$ for $t> 0$,
which satisfies $sh(t/s)=tr_1(s/t)$ for $r_1(u)=r(u-1)$)
appears in many exponential concentration inequalities in the literature,
including Bennett's inequality \cite{Ben62} (see also \cite{BLM13:book}),
and empirical process theory \cite{Wel78}. It would be nice to have a clearer understanding
of why these functions appear in so many related contexts even though the specific
circumstances vary quite a bit.
\end{rem}

\begin{rem}
Note that the function $r$ is convex on $\R$ and  has a quadratic behavior in the neighborhood of $0$ ($r(u)\sim_0\frac{u^2}{2}$) and a linear behavior at $+\infty$ ($r(u)\sim_\infty u$). 
\end{rem}

\begin{cor}\label{cor:cor-mainconcentrate}
With the assumptions and notation of Theorem~\ref{thm:mainconcentrate},
we have for any $t>0$ that
\ben\begin{split}
\mathbf{P}\{ \h(X)-h(X) \geq t\} \,\, &\leq \exp\bigg\{-K r\bigg(\frac{t}{K}\bigg)\bigg\} \\
\mathbf{P}\{ \h(X)-h(X) \leq -t\} & \leq  \exp\bigg\{-K r\bigg(-\frac{t}{K}\bigg)\bigg\} 
\end{split}\een
\end{cor}

The proof is classical and often called the Cram\'er-Chernoff method (see for example section 2.2 in \cite{BLM13:book}). It uses the Legendre transform $\varphi^*$ of a convex function $\varphi :\R\to\R\cup\{+\infty\}$ defined for $y\in\R$ by 
$$\varphi^*(y)=\sup_xxy-\varphi(x).$$
Notice that if $\min\varphi=\varphi(0)$ then for every $y>0$, the supremum is reached at a positive $x$, that is 
$\varphi^*(y)=\sup_{x>0}xy-\varphi(x).$ Similarly, for $y<0$, the supremum is reached at a negative $x$.

\begin{proof}
The idea is simply to use Markov's inequality in conjunction with Theorem~\ref{thm:mainconcentrate},
and optimize the resulting bound. 

For the lower tail, we have for $\beta>0$ and $t>0$,
\bee\label{eq:lowertail}\begin{split}
\mathbb{P}[\h(X)-h(X)\leq -t]
&\leq \mathbf{E}\bigg[e^{-\beta\big(\h(X)-h(X)\big)}\bigg]e^{-\beta t} \\
&\leq \exp\bigg\{K\bigg(r(\beta)-\frac{\beta t}{K}\bigg)\bigg\} .
\end{split}\ene
Thus minimizing on $\beta>0$, and using the remark before the proof, we get 
\be\label{eq:legendre}
\mathbb{P}[\h(X)-h(X)\leq -t]\le \exp\bigg\{-K\sup_{\beta>0}\bigg(\frac{\beta t}{K}-r(\beta)\bigg)\bigg\}
=e^{-Kr^*\left(\frac{t}{K}\right)} .
\ee
Let us compute the Legendre transform $r^*$ of $r$. For every $t$, one has 
$$r^*(t)=\sup_u tu-r(u)=\sup_{u>-1}\left(tu-u+\log(1+u)\right).$$
One deduces that $r^*(t)=+\infty$  for $t\ge1$. For $t<1$, by differentiating, the supremum is reached at $u=t/(1-t)$ and replacing in the definition we get 
$$r^*(t)=-t-\log(1-t)=r(-t).$$ 
Thus $r^*(t)=r(-t)$ for all $t\in\R$. Replacing, in the inequality (\ref{eq:legendre}), we get the result for the lower tail.

For the upper tail, we use the same argument:  for $\beta>0$ and $t>0$,
\bee\label{eq:uppertail}\begin{split}
\mathbb{P}[\h(X)-h(X)\geq t]
&\leq \mathbf{E}\bigg[e^{\beta\big(\h(X)-h(X)\big)}\bigg]e^{-\beta t} \\
&\leq \exp\bigg\{K\bigg(r(-\beta)-\frac{\beta t}{K}\bigg)\bigg\} .
\end{split}\ene
Thus minimizing on $\beta>0$, we get 
\be\label{eq:legendre2}
\mathbb{P}[\h(X)-h(X)\geq t]\le \exp\bigg\{-K\sup_{\beta>0}\bigg(\frac{\beta t}{K}-r(-\beta)\bigg)\bigg\}.
\ee
Using the remark before the proof, in the right hand side term appears the Legendre transform of the function $\tilde{r}$ defined by $\tilde{r}(u)=r(-u)$. Using that $r^*(t)=r(-t)=\tilde{r}(t)$, we deduce that $(\tilde{r})^*=(r^*)^*=r$. Thus the inequality (\ref{eq:legendre2}) gives the result for the upper tail.

\end{proof}

\section{Conclusion}
\label{sec:opt-dev}

The purpose of this section is to combine the results of Sections~\ref{sec:var}
and \ref{sec:gen} to deduce sharp bounds for the moment generating function 
of the information content of random vectors with log-concave densities.
Naturally these yield good bounds on the deviation probability of the information content $\h(X)$
from its mean $h(X)=\E \h(X)$. 
We also take the opportunity to record some other easy consequences.

\begin{thm}\label{thm:lc-dev}
Let $X$ be a random vector in $\mathbb{R}^n$ with a log-concave density $f$. For $\beta<1$,
\bee
\E\bigg[e^{\beta[\h(X)-h(X)]}\bigg]\leq \E\bigg[e^{\beta[\h(X^*)-h(X^*)]}\bigg],
\ene
where $X^*$ has density $f^*=e^{-\sum_{i=1}^nx_i}$, restricted to the positive quadrant.
\end{thm}

\begin{proof}
Taking $K=n$ in Theorem~\ref{thm:mainconcentrate} (which we can do in the log-concave setting
because of Theorem~\ref{thm:varent}), we obtain:
\ben
\E\big[e^{\beta\{\h(X)-h(X)\}}\big] \leq e^{nr(-\beta)} , \quad \beta\in\RL .
\een
Some easy computations will show:
\ben
\E\big[e^{\beta\{\h(X^*)-h(X^*)\}}\big]=e^{nr(-\beta)} , \quad \beta\in\RL .
\een
This concludes the proof. 

\end{proof}

As for the case of equality of Theorem~\ref{thm:varent}, discussed in Remark~\ref{rem:extremal}, notice that 
there is a broader class of densities for which one has equality in Theorem~\ref{thm:lc-dev},
including all those of the form $e^{-\|x\|_K}$, where $K$ is a symmetric convex body.

\begin{rem}
The assumption $\beta<1$ in Theorem~\ref{thm:lc-dev} is strictly not required;
however, for $\beta \geq 1$, the right side is equal to $+\infty$. Indeed, already for $\beta=1$,
one sees that for any random vector $X$ with density $f$,
\ben\begin{split}
\E\big[e^{\h(X)-h(X)}\big] &= e^{-h(X)} \E\bigg[ \frac{1}{f(X)}\bigg] = e^{-h(X)} \int_{\text{supp(f)}} dx \\
&=  e^{-h(X)} \text{Vol}_n (\text{supp(f)}) ,
\end{split}\een
where $\text{supp(f)}=\overline{\{x\in\RL^n: f(x)>0\}}$ is the support of the density $f$ and $\text{Vol}_n$ denotes Lebesgue measure on $\RL^n$.
In particular, this quantity for $X^*$, whose support has infinite Lebesgue measure, is $+\infty$.
\end{rem}

\begin{rem}
Since
\bee
\lim_{\alpha\rightarrow 0}\frac{2}{\alpha^2}\mathbf{E}\bigg[e^{\alpha(\log f(X)-\mathbf{E}[\log f(X)])}\bigg] =V(X),
\ene
we can recover Theorem~\ref{thm:varent} from Theorem~\ref{thm:lc-dev}.
\end{rem}

Taking $K=n$ in Corollary~\ref{cor:cor-mainconcentrate} (again
because of Theorem~\ref{thm:varent}), we obtain:

\begin{cor}\label{cor:specialconcentrate}
Let $X$ be a random vector in $\mathbb{R}^n$ with a log-concave density $f$. For $t>0$,
\bee
&\mathbb{P}[\h(X)-h(X)\leq -nt]\leq e^{-nr(-t)},\\
&\mathbb{P}[\h(X)-h(X)\geq nt]\leq e^{-nr(t)},
\ene
where $r(u)$ is defined in Theorem~\ref{thm:mainconcentrate}.
\end{cor}


The original concentration of information bounds obtained in \cite{BM11:aop} were suboptimal
not just in terms of constants but also in the exponent; specifically it was proved there that
\be\label{eq:subopt}
\P\left\{\nth\,\big|\h(X) - h(X)\big| \geq t \,\right\} \leq 2\,e^{-ct\sqrt{n}}
\ee
for a universal constant $c>1/16$ (and also that a better bound with $ct^2 n$ in the exponent
holds on a bounded range, say, for $t\in (0,2]$).
One key advantage of the method presented in this paper, apart from its utter simplicity, is
the correct linear dependence of the exponent on dimension. Incidentally, we learnt from a lecture of
B. Klartag \cite{Kla15:ima} that another proof of \eqref{eq:subopt} can be given based on the
concentration property of the eigenvalues of the Hessian of the Brenier map
(corresponding to  optimal transportation from one log-concave density to another)
that was discovered by Klartag and Kolesnikov \cite{KK15};
however, the latter proof shares the suboptimal $\sqrt{n} t$ exponent of \cite{BM11:aop}.

The following inequality is an immediate corollary of Corollary~\ref{cor:specialconcentrate}
since it merely expresses a bound on the support of the distribution of the information content.

\begin{cor}\label{cor:bmrenyi}
Let $X$ have a log-concave probability density function $f$ on $\mathbb{R}^n$. Then:
\bee
h(X)\leq -\log \|f\|_{\infty}+n.
\ene
\end{cor}
\begin{proof}
By Corollary~\ref{cor:specialconcentrate}, almost surely,
\bee
\log f(X) \leq\mathbf{E}[\log f(X)]+n,
\ene
since when $t\geq 1$, $\mathbb{P}[\log f(X)-\mathbf{E}[\log f(X)]\geq nt]=0$.
Taking the supremum over all realizable values of $X$ yields
\bee
\log \|f\|_\infty \leq\mathbf{E}[\log f(X)]+n,
\ene
which is equivalent to the desired statement.
\end{proof}

Corollary~\ref{cor:bmrenyi} was first explicitly proved in \cite{BM11:it}, where several applications
of it are developed, but it is also implicitly contained in earlier work (see, e.g., the proof of Theorem 7 in \cite{FM08:1}).

An immediate consequence of Corollary~\ref{cor:bmrenyi}, unmentioned in \cite{BM11:it}, is a result due to \cite{Fra97}:

\begin{cor}
Let $X$ be a random vector in $\mathbb{R}^n$ with a log-concave density $f$. Then
\bee
\|f\|_{\infty}\leq e^nf(\mathbf{E}[X]).
\ene
\end{cor}
\begin{proof}
By Jensen's inequality,
\bee
\log f(\E X)\geq\E[\log f(X)].
\ene
By Corollary~\ref{cor:bmrenyi},
\bee
\E[\log f(X)]\geq \log \|f\|_{\infty}-n.
\ene
Hence,
\bee
\log f(\E X)\geq\log \|f\|_{\infty}-n.
\ene
Exponentiating concludes the proof.
\end{proof}

Finally we mention that the main result may also be interpreted as a small ball inequality
for the random variable $f(X)$. As an illustration, we record a sharp form of \cite[Corollary 2.4]{KM05}
(cf., \cite[Corollary 5.1]{Kla07:1} and \cite[Proposition 5.1]{BM12:jfa}).

\begin{cor}
Let $f$ be a log-concave density on $\mathbb{R}^n$. Then
\bee
\mathbb{P}\{f(X)\geq c^n\|f\|_\infty\} \geq 1-\bigg(e \cdot c\cdot\log\bigg(\frac{1}{c}\bigg)\bigg)^n,
\ene
where $0<c<\frac{1}{e}$.
\end{cor}
\begin{proof}
Note that
\bee\begin{split}
\mathbb{P}\{f(X)\leq c^n\|f\|_\infty\}
&=\mathbb{P}\{\log f(X)\leq\log \|f\|_\infty+n\log c\}\\
&= \mathbb{P}\{\h(X)\geq -\log \|f\|_\infty-n\log c\} \\
&\leq \mathbb{P} \{\h(X) \geq h(X)-n(1+\log c)\}.
\end{split}\ene
using Corollary~\ref{cor:bmrenyi} for the last inequality. Applying
Corollary~\ref{cor:specialconcentrate} with $t=-\log c -1$ yields
\ben
\mathbb{P}\{f(X)\leq c^n\|f\|_\infty\}
\leq e^{-nr(-1-\log c)} .
\een
Elementary algebra concludes the proof.
\end{proof}

Such ``effective support'' results are useful in convex geometry as they can allow to reduce
certain statements about log-concave functions or measures to statements about convex sets;
they thus provide an efficient route to proving functional or probabilistic analogues of known
results in the geometry of convex sets. Instances where such a strategy is used include
\cite{KM05, BM12:jfa}. These and other applications of the concentration of information
phenomenon are discussed in \cite{MWB15}.


\subsection*{Acknowledgment}
We are indebted to Paata Ivanisvili, Fedor Nazarov, and Christos Saroglou for useful comments on an earlier
version of this paper. In particular, Christos Saroglou pointed out that the class of extremals
in our inequalities is wider than we had realized, and Remark~\ref{rem:extremal} is due to him.
We are also grateful to Fran\c{c}ois Bolley, Dario Cordero-Erausquin and an anonymous referee for pointing out relevant references.


\end{document}